\documentclass[12pt,a4paper]{article} 
\usepackage{setspace}
\usepackage[utf8]{inputenc} 
\usepackage{amsmath}
\usepackage{amsthm} 
\usepackage{amssymb}
\usepackage{todonotes}
\usepackage{xcolor}
\usepackage{graphicx}
\usepackage{tikz}
\usepackage[hidelinks]{hyperref}
\usepackage[nameinlink]{cleveref}
\usepackage[procnames]{listings}
\usepackage[margin=1.2in]{geometry}

\newtheorem{theorem}{Theorem}
\newtheorem{prop}{Proposition}
\newtheorem{lemma}{Lemma}

\newtheorem{claim}{Claim}
\theoremstyle{definition}
\newtheorem{rmk}{Remark}
\newtheorem{question}{Question}

\newtheorem{definition}{Definition}

\def\EE{\mathbb{E}}

\def\II{\mathbb{I}}

\def\NN{\mathbb{N}}

\def\PP{\mathbb{P}}

\def\RR{\mathbb{R}}

\def\ZZ{\mathbb{Z}}

\def\calF{\mathcal{F}}

\def\calP{\mathcal{P}}

\newcommand{\abs}[1]{\left\lvert#1\right\rvert}

\newcommand{\norm}[1]{\left\lVert#1\right\rVert}

\newcommand{\floor}[1]{\left\lfloor #1 \right\rfloor}
\newcommand{\ceil}[1]{\left\lceil #1 \right\rceil}
\newcommand{\paren}[1]{\left( #1 \right)}
\newcommand{\sqb}[1]{\left[ #1 \right]}
\newcommand{\set}[1]{\left\{ #1 \right\}}

\newcommand{\Var}{\mathrm{Var}}

\title{Sharp Threshold for the Convergence of Nonstationary Averaging}
\author{Saba Lepsveridze,  Elchanan Mossel}
\date{}

\begin{document} 
\maketitle

\begin{abstract}
We study non-stationary averaging processes, where each term of a sequence is a weighted average of previous terms, namely $a_{n+1} = \sum_{j=1}^n p_n(j) a_j$. Our results extend classical theory in two distinct regimes. First, we prove a sharp threshold for convergence in the regime where the weights are bounded between two envelopes $(\log n)^{-\alpha} \le np_n(\cdot) \le (\log n)^{\beta}$. We show that the sequence necessarily converges when $\alpha + \beta / 2 \leq 1$, while $\alpha + \beta / 2 > 1$ the convergence can fail.  Second, we study complementary fixed shape regime, when $p_n$ is obtained by a fixed limiting density on $(0,1)$. We show that under mild regularity assumptions, the sequence converges.
\end{abstract}

\section{Introduction}

Weighted running averages show up frequently in probability, dynamics, and optimization. A basic primitive behind these procedures is that the next iterate is a weighted average of the past. While classical theory understands these processes well when the weights are constant or uniform, much less is known when the averaging weights could be non-stationary and potentially spiky.

Let $\set{p_n}_{n = 1}^\infty$ be a sequence of probability measures with  $p_n \in \calP([n])$.  Consider a sequence of vectors  $\set{a_n}_{n = 1}^\infty \subset \RR^d$ defined by
 \begin{equation*}
	a_{n+1} = \EE_{j \sim p_n} a_j = \sum_{j = 1}^n p_n(j) a_j \text{ for } n \geq k,
\end{equation*}
and an arbitrary initialization $a_1, \ldots, a_k \in \RR^d$. We ask: 
\begin{question}
	Under what conditions on $\set{p_n}_{n = 1}^\infty$ does the sequence  $\set{a_n}_{n = 1}^\infty$ necessarily converge?  
\end{question}
A natural global constraint is to bound each weight between two envelopes
\begin{equation*}
	\frac{f(n)}{n}\leq p_n(j) \leq \frac{c(n)}{n} \text{ for all } j \in [n],
\end{equation*}
where $f(n)$ and $c(n)$ describe a floor and a ceiling, respectively. This allows a decay of $f(n)$ and blow-up of $c(n)$, which is far from uniform averaging. We answer this question by obtaining optimal rates for $f$ and $c$ under which convergence holds. 

\begin{theorem}\label{thm:main}
Suppose $f(n) = A (\log n)^{-\alpha} $ 	and $c(n) = B(\log n)^{\beta}$ with $\alpha, \beta > 0$.
\begin{itemize}
	\item[-] If $\alpha + \beta / 2 \leq 1$, then for any choice of $\set{p_n}_{n = 1}^\infty$ and any initialization, the sequence $\set{a_n}_{n=1}^\infty$ necessarily converges.
	\item[-] If $\alpha + \beta / 2 > 1$, then there exists a choice of $\set{p_n}_{n = 1}^\infty$ and an initialization such that $\set{a_n}_{n=1}^\infty$ fails to converge.
\end{itemize}
\end{theorem}

In addition to the worst-case envelope-bounded weights considered above, \cref{sec:fixed-shape} establishes \cref{thm:discrete-fixed-shape} for a complementary \emph{fixed-shape} regime. In this setting, the weights are obtained by discretizing a limiting probability density on $(0,1)$. The additional structure permits a renewal-theoretic analysis and yields convergence under suitable regularity and approximation assumptions.

\begin{theorem}[Informal]
Let $\{p_n\}_{n=1}^\infty$ be a sequence of discrete probability distributions that approximate a density $p:(0,1)\to[0,\infty)$ with finite logarithmic moment. Let $\{a_n\}_{n=1}^\infty \subset \RR^d$ satisfy
\begin{equation*}
a_{n+1} = \EE_{j\sim p_n}\!\bigl[a_j\bigr]
\qquad \text{for all } n \ge k .
\end{equation*}
Then the sequence $\{a_n\}_{n=1}^\infty$ converges.
\end{theorem}

The precise assumptions and statement appear in \cref{thm:discrete-fixed-shape}.

\subsection{Related Work}

Note that the recursion $a_{n+1}=\sum_{j=1}^n p_n(j)a_j$ can be viewed  as $a_{n+1}=\EE[a_{J_n}]$, where $J_n\sim p_n$.  Such recursions often appear in probabilistic models with memory as detailed below. These works motivate our setting but we assume neither stationarity nor model specific structure. In our paper, we focus on two general regimes: \cref{thm:main} gives a sharp worst case convergence threshold over all nonstationary kernels obeying pointwise envelope bounds, while \cref{thm:discrete-fixed-shape} treats a structured scaling invariant fixed shape regime. 

\paragraph{Split trees and tagged lineages.}
A concrete way our recurrence shows up in random trees is through the standard tree = root + subtrees decomposition.  In a split tree with $n$ items, the root partitions the items into subtree sizes $(J_{n,1},\dots,J_{n,b})$ according to a random split, and many parameters are obtained by iterating this decomposition down the tree as in \cite{janson2019splittrees}.  If one follows a tagged object (for example an uniformly chosen item, or the lineage of a random leaf), then at each split one keeps only the unique child subtree containing the tag. Hence, for many scalar quantity $a_n$ that is a function of the tagged subtree, one obtains
\[
a_{n+1}=\EE[a_{J_n}]=\sum_{j\le n}\PP(J_n=j)\,a_j.
\]
This recursion is often the key input before lifting to global statements about the whole tree like LLNs/CLTs for depths, profiles, and related additive functionals.

In fact, our fixed shape \cref{thm:discrete-fixed-shape} can be applied as a black box to Aldous' $\beta$-splitting model \cite{aldous1996cladograms} to obtain convergence behaviour for many recursions in these tree models for $\beta > -1$, when the splitting distribution has finite moments. 

Long-memory walks such as the elephant random walk also generate full-memory averaging recurrences, but we mention them only as further motivation rather than as a structural input \cite{schutz2004elephant,baur2016elephant}.

\paragraph{Self-averaging sequences}
Oftentimes  "probability of an event at time $n$" problems can be written in a self averaging form (see group Russian roulette example \cite{cator2017self-averaging} and many other references therein): one identifies some bounded statistic or an event probability $a_n$ and a random lookback index $J_n\in[n]$ such that
\[
a_{n}=\EE[a_{J_n}].
\]
Setting $p_n(j)=\PP(J_n=j)$, we reduce to studying our recursion.  Cator and Don \cite{cator2017self-averaging} study precisely this equation for bounded sequences, but under a concentration hypothesis $J_n\approx \alpha n$ with $\Var(J_n)=O(n)$.  Our results complement this line of research by studying worst-case adversarial envelope and a structured scaling regimes. In fact, our divergence construction \cref{thm:divergence} shows that oscillations can persist even when most of the mass spread over polylogarithmic number of indices.

\paragraph{Renewal theory}
Renewal theory analyzes convolution recursions such as
\[
u_0=1
\qquad\text{and}\qquad
u_n=\sum_{k=1}^n f_k\,u_{n-k},
\]
where $(f_k)_{k\ge1}$ is a probability mass function and $(u_n)$ is the associated renewal sequence \cite{erdos1949renewal,chung1952renewal,serfozo2009basics}.  A classical result by Erd\H{o}s-Feller-Pollard theorem states that when the mean $\mu\triangleq \sum_{k\ge1} k f_k$ is finite and $f$ is aperiodic, one has the sharp limit
\[
u_n \longrightarrow \frac{1}{\mu},
\]
which is the discrete-time analogue of the "renewal density $\to 1/\mu$" principle \cite{erdos1949renewal}.  More generally, key renewal theorems describe limits for perturbed renewal equations and identify the limiting constant via the mean $\mu$ and an explicit overshoot law \cite{chung1952renewal,serfozo2009basics}.

Our fixed shape regime in \cref{sec:fixed-shape} can be put in this framework after a logarithmic change of variables.  Roughly speaking, when the weights come from discretizing a density on $(0,1)$ the recursion has form $F(x)\approx \EE[F(Tx)]$ with $T\in(0,1)$.  Writing $x=e^{s}$ and $G(s)=F(e^{s})$ turns this into an additive renewal equation
\[
G(s)=\EE\!\left[G(s-Y)\right]+\eta(s),
\qquad
Y\triangleq \log(1/T),
\]
where $\eta$ captures discretization error.  Thus a finite log moment $\EE[\log(1/T)]<\infty$ plays exactly the role of a finite mean increment in classical renewal theory, and yields convergence together with an explicit residual description. This is stated in \cref{lem:renewal-with-error}, and might be of independent use beyond this application.

% \paragraph{Slow variation heuristics.}
% Thinking about the fixed shape regime $a_{n+1}\approx \EE[a_{\lfloor Tn\rfloor}]$, it is tempting to try and control the difference $a_{n+1}-a_n$ by invoking Ash-Erd\H{o}s-Rubel theory of $\varphi$-slowly varying functions \cite{ash1974slowly}. However, it turns out that a basic implementation does not obtain the convergence here as this method ignores cancellation coming from the averaging procedure. 

\paragraph{Consensus/social learning and stochastic approximation.}
Classical distributed consensus and social learning models iterate stochastic matrices on a fixed agent set and analyze convergence via connectivity/mixing assumptions \cite{degroot1974reaching,olshevsky2011convergencespeed}. 
In the Bayesian setting Bala and Goyal~\cite{bala1998learning} consider a Bayesian model, where when an agent joins they observe all previous agents (and their private signal) before taking their action. 
The analogous model in the DeGrott framework consists of growing network where when agents arrive sequentially, and agent $n+1$ forms an opinion by averaging earlier opinions with an attention profile $p_n$,
\[
x_{n+1}=\sum_{j\le n} p_n(j)\,x_j.
\] 

It is natural to ask when is asymptotic consensus reached in such a model. 
The Erd\H{o}s-Feller-Pollard theorem provides a positive answer when $p(n,j) = f_{n-j}$ and $f$ has a finite mean and is non-periodic. Our results show that consensus is achieved even if agents use different averaging weights as long as these are equitable enough among the preceding agents. Since our results are tight they also provide examples where if the weights are not equitable consensus is not reached. Prior work in learning on networks highlighted the role of various notions of equability in reaching consensus and in learning~\cite{golub2010naive,acemoglu2011bayesian,mossel2015strategic} 

\subsection{Acknowledgements}
E.M. Is partially supported by ARO MURI N00014241274, by Vannevar Bush Faculty Fellowship ONR-N00014-20-1-2826 and by a Simons Investigator Award.

S.L.'s research supported in part by NSF--Simons collaboration grant DMS-2031883.

\section{Reduction}
To prove the threshold in \cref{thm:main}, we reduce the problem to a one dimensional extremal process. This reduction starts by applying the argument coordinatewise and using affine invariance of the recursion. In particular, we assume without loss of generality that $\set{a_n}_{n=1}^\infty \subset \RR$. By affine invariance of the sequence, we can shift and scale the initialization so that all terms are contained in the interval $[0,1]$. For convenience, we also reparametrize the problem by setting
\begin{equation} \label{eq:repar}
	\varepsilon_n \triangleq  f(n) \text{ and } \delta_n \triangleq  \frac{1-f(n)}{c(n) - f(n)}
\end{equation}

The parameter $\delta_n$ is chosen so that assigning maximal weight $c(n)/n$ to the largest  $\delta_n n$ terms\footnote{To improve readability, we will ignore integrality issues throughout the paper, since they can be handled with routine adjustments and do not affect the asymptotic arguments.} and $f(n)/n$ to the remaining terms yields a valid probability measure.

We now introduce important notation.
For a sequence $\set{a_n}_{n = 1}^\infty$ and an index $m \in \NN$ denote
\begin{equation*}
	a_1^{(m)} \geq \cdots \geq  a_m^{(m)} \text{ and } a_1^{[m]} \leq \cdots \leq  a_m^{[m]}
\end{equation*}  
to be the descending and ascending orderings of the partial sequence $\set{a_1, \ldots, a_m}$, respectively. Similarly, we define the percentile averages
\begin{gather*}
	\bar{a}_m \triangleq  \frac{a_1 + \cdots + a_m}{m} \text{ and }
	\bar{a}^{(m)}_j \triangleq \frac{a_1^{(m)} + \cdots + a_j^{(m)}}{j} \text{ and }
	 \bar{a}^{[m]}_j \triangleq \frac{a_1^{[m]} + \cdots + a_j^{[m]}}{j}.
\end{gather*}
In words, $\bar{a}_m$ is the average of all terms $a_1, \ldots, a_m$, while $\bar{a}^{(m)}_j$ and $\bar{a}^{[m]}_j $ are averages of the largest and smallest $j$ terms among the first $m$ elements, respectively. 

\medskip

We now describe an equivalent formulation of the problem that we will work with for the rest of the paper.

\begin{lemma}
	Let $\set{a_n}_{n = 1}^\infty$ be a sequence of real numbers and $\set{p_n}_{n = 1}^\infty$ be a sequence of probability measures with $p_n$ supported on $[n]$ such that 
	\begin{equation}\label{eq:og-formulation}
		a_{n+1} = \EE_{j \sim p_n} a_j \text{ for all } n\geq k\text{ and } \frac{f(n)}{n} \leq p_n(j) \leq \frac{c(n)}{n} \text{ for all } j \in [n].
	\end{equation}
	Then, for all $n \geq k$ the sequence $\set{a_n}_{n = 1}^\infty$  satisfies 
	\begin{equation}\label{eq:new-formulation}
		a_{n+1} \in \varepsilon_n \bar{a}_n + (1-\varepsilon_n) [\bar{a}_{\delta_n n}^{[n]}, \bar{a}_{\delta_n n}^{(n)}] \text{ for all } n \geq k.
	\end{equation}
	Furthermore, if $\set{a_n}_{n=1}^\infty$ satisfies (\ref{eq:new-formulation}), then there exist $\set{p_n}_{n=1}^\infty$ such that (\ref{eq:og-formulation}) holds.
\end{lemma}

\begin{proof}
Note that $\varepsilon_n$ and $\delta_n$ are chosen precisely so that  
\begin{equation*}
	\frac{c(n)}{n}  \times \delta_nn  + \frac{f(n)}{n}\times (n - \delta_n n) = 1 \text{ and } f(n) = \varepsilon_n.
\end{equation*}
Suppose first that $\set{a_n}_{n = 1}^\infty$ and $\set{p_n}_{n = 1}^\infty$ satisfy  (\ref{eq:og-formulation}). Then, 
\begin{equation*}
	a_{n+1} = \sum_{j = 1}^n p_n(j) a_j \leq \sum_{j = 1}^{\delta_n n} \frac{c(n)}{n}a_{j}^{(n)} + \sum_{j = \delta_n n + 1}^n\frac{f(n)}{n} a_j^{(n)} = (1-\varepsilon_n)\bar{a}_{\delta_n n}^{(n)} + \varepsilon_n\bar{a}_n
\end{equation*}
In words, since the objective is linear in $\set{p_n}$, the maximum is attained by assigning maximal weight to the largest entries. Analogously, we have  
\begin{equation*}
	a_{n+1} \geq (1-\varepsilon_n)\bar{a}_{\delta_n n}^{[n]} + \varepsilon_n \bar{a}_n.
\end{equation*}
On the other hand, suppose $\set{a_n}_{n=1}^\infty$ satisfies (\ref{eq:new-formulation}). Then, for all $n \geq k$ there exists $\lambda_n \in [0,1]$, such that
\begin{equation*}
	a_{n+1} = \varepsilon_n \bar{a}_n + (1-\varepsilon_n)(\lambda_n \bar a_{\delta_n n}^{(n)} + (1-\lambda_n)a_{\delta_n n}^{[n]}).
\end{equation*}
Hence, we can write $a_{n+1} = \EE_{j \sim p_n} a_j$, where
\begin{equation*}
	p_n(j) = \frac{f(n)}{n} + \frac{c(n)-f(n)}{n}
		 \begin{cases}
		   \lambda_n  & \text{ if } a_j \text{ is among the largest } \delta_n n \text{ terms } \\
		1-\lambda_n & \text{ if } a_j \text{ is among the smallest } \delta_n n \text{ terms }\\
		0 &\text{ otherwise } \\
	\end{cases}
\end{equation*}
This concludes the two way reduction.
\end{proof}
With the reparameterization (\ref{eq:repar}), we abuse notation and write 
\begin{equation*}
	\varepsilon_n = A(\log n)^{-\alpha} \text{ and } \delta_n = B(\log n)^{-\beta},
\end{equation*}
for some constants $A, B >0$.

\section{Majorization}

In this section, we develop a comparison tool that helps us lift the analysis from simple sequences to more complex ones. In particular, we use majorization to prove the main technical propositions in the next section.

\begin{definition}(Majorization)
	We say that $\set{a_n}_{n = 1}^k$ majorizes  $\set{b_n}_{n = 1}^k$ if 
	\begin{gather*}
		a_{1}^{(k)} \geq b_1^{(k)} \\
		a_1^{(k)} + a_2^{(k)} \geq b_1^{(k)} +b_2^{(k)} \\ 
		\cdots \\
		a_1^{(k)} + \cdots +  a_k^{(k)} \geq b_1^{(k)} +  \cdots + b_k^{(k)} 
	\end{gather*}
	We denote this relation by $\set{a_n}_{n = 1}^k \succeq \set{b_n}_{n = 1}^k$.
\end{definition}

We note that this definition differs from the typical notion of majorization as we do not require equality of total sums. The following lemma states that if the initialization of one sequence majorizes that of another, then the former dominates the latter at all future times.

\begin{lemma}[Majorization]\label{lem:majorization}
	Suppose $\set{a_n}_{n = 1}^k$ majorizes $\set{b_n}_{n = 1}^k$ and 
    \begin{equation*}
		a_{n+1} = \varepsilon_n \bar{a}_n +(1-\varepsilon_n)\bar{a}_{\delta_n n}^{(n)} \text{ and } b_{n+1} \leq  \varepsilon_n \bar{b}_n +(1-\varepsilon_n)\bar{b}_{\delta_n n}^{(n)} \text{ for all } n\geq k.
	\end{equation*}
%	such that $\set{a_n}_{n = 1}^k \succeq \set{b_n}_{n = 1}^k$. 
Then
	\begin{enumerate}
		\item $\set{a_n}_{n = 1}^m \succeq \set{b_n}_{n = 1}^m$ for all $m \geq k$, and  
		\item $a_m \geq b_m$ for all $m \geq k+1 $. 
	\end{enumerate}
	
	\end{lemma}

\begin{proof} 
	We proceed to prove by induction. Base case $m = k$ holds by assumption. Assume now that $\set{a_n}_{n = 1}^m \succeq \set{b_n}_{n = 1}^m$. Majorization implies that  $\bar{a}_{j}^{(m)} \geq \bar{b}_{j}^{(m)}$ holds for any $j \in  \set{1, \ldots, m}$. In particular, 
		\begin{equation*}
		a_{m+1} = \varepsilon_m \bar{a}_m + (1-\varepsilon_m)\bar{a}_{\delta_m m}^{(m)}  \geq \varepsilon_m \bar{b}_m + (1-\varepsilon_m)\bar{b}_{\delta_m m}^{(m)} \geq b_{m+1}.
	\end{equation*}
	It remains to show that $\set{a_n}_{n = 1}^{m+1} \succeq \set{b_n}_{n = 1}^{m+1}$. Define multisets
	\begin{equation*}
		A_{j}^{(m)} \triangleq \set{a_1^{(m)}, \ldots, a_j^{(m)}} \text{ and } B_{j}^{(m)} \triangleq \set{b_1^{(m)}, \ldots, b_j^{(m)}}.
	\end{equation*}
	For a multiset $S$ we define $\Sigma S \in \RR$ to be the sum of the terms within.  We consider two cases as follows.
	\begin{itemize}
		\item[-] If $b_{m+1} \notin B_{j}^{(m)}$, then 
		\begin{equation*}
			\Sigma  B_{j}^{(m+1)}=\Sigma B_j^{(m)} \leq \Sigma A_j^{(m)} \leq \Sigma A_{j}^{(m+1)}.
		\end{equation*}
		\item[-] If $b_{m+1} \in B_{j}^{(m)}$, then 
		\begin{equation*}
			\Sigma  B_{j}^{(m+1)}= b_{m+1} +\Sigma B_{j-1}^{(m)} \leq a_{m+1} + \Sigma A_{j-1}^{(m)} \leq  \Sigma A_{j}^{(m+1)}.
		\end{equation*}
	 	\end{itemize}
	 This concludes the proof.	
\end{proof}

Another useful observation is that even if one sequence initially majorizes another, this dominance need not manifest itself in future iterates. More precisely, suppose the initialization of the second sequence is obtained from that of the first by replacing the top $m$ terms with their average. Then the two sequences evolve identically for as long as newly generated terms do not exceed the $m$-th largest term in $a$.

\begin{lemma}(Reverse Majorization)\label{lem:reverse-majorization} Let $\delta_n$ be a decaying parameter such that the sequence $\set{\delta_n n}_{n = k}^\infty$ is non-decreasing. Suppose $\set{a_n}_{n = 1}^\infty$ and $\set{b_n}_{n = 1}^\infty$ satisfy 
	\begin{align*}
		a_{n+1} &= \varepsilon_n \bar{a}_n +(1-\varepsilon_n)\bar{a}_{\delta_n n}^{(n)}  \\
		b_{n+1} &= \varepsilon_n \bar{b}_n +(1-\varepsilon_n)\bar{b}_{\delta_n n}^{(n)} \text{ for all } n \geq k.
	\end{align*}
	Suppose that there is $m \leq \delta_k k$, such that 
	\begin{equation*}
		b_j^{(k)} = \bar{a}_m^{(k)} \text{ for all } j \leq m \text{ and } b_j^{(k)} = a_j^{(k)} \text{ for all } m < j \leq k;  
	\end{equation*}
	Then, $a_n = b_n$ for all $n > k$ as long as $b_n \leq a_m^{(k)}$ for all $n  > k$. 
\end{lemma}
\begin{proof}
The main idea is to show that the sequences $a$ and $b$ evolve identically: the only discrepancy between them is confined to the top $m$ terms, and these terms remain at the top throughout.

	It suffices to show $\bar{a}_n = \bar{b}_n$ and $\bar{a}_{\delta_n n}^{(n)} = \bar{b}_{\delta_n n}^{(n)}$ for all $n \geq k$. We will proceed to prove by induction that for all $n \geq k$,
\begin{equation*}
	b_j^{(n)} = \bar{a}_m^{(n)} \text{ for all } j \leq m \text{ and } b_j^{(n)} = a_j^{(n)} \text{ for all } m < j \leq k.
\end{equation*}
	Note that since $m \leq \delta_k k \leq \delta_n n$, this implies both  $\bar{a}_{\delta_n n}^{(n)} = \bar{b}_{\delta_n n}^{(n)}$ and $\bar{a}_n = \bar{b}_n$.
	
	The base case $n = k$ holds by assumption. On the other hand, inductive hypothesis immediately implies $a_{n+1} = b_{n+1}$. Moreover, by assumption, $b_{n+1} \leq a_m^{(k)} = a_m^{(n)}$. This means that the new term does not enter the top $m$ terms in $a$ or $b$, which concludes the proof.
\end{proof}

\section{Main Technical Tools}
In this section, we use majorization to establish two main technical results, \cref{prop:technical-convergence} and \cref{prop:technical-divergence}, which will be used in the next section to prove convergence and non-convergence, respectively.

Roughly speaking, \cref{prop:technical-convergence} asserts that if the initialization of the sequence is mostly contained in an interval $[B,U]$ and its average is close to $B$, then the future iterates remain uniformly bounded away from $U$.

\begin{prop}\label{prop:technical-convergence} Suppose $\varepsilon_n = A(\log n)^{-\alpha}$ and $\delta_n = B(\log n)^{-\beta}$ with $\alpha, \beta \geq 0$. There exist constants $c$ and $C$ such that the following holds. Suppose $k \geq C$  and let $\set{a_n}_{n=1}^\infty$ be a sequence with initialization $a_1, \ldots, a_k \in [0,1]$  satisfying
\begin{equation*}
		a_{n+1} = \varepsilon_n \bar{a}_n + (1-\varepsilon_n)\bar{a}^{(n)}_{\delta_n n} \text{ for all } n \geq k.
\end{equation*} 
Suppose there exists an interval $[B,U]$ such that 
 \begin{enumerate} 
 	\item $\bar{a}_k =  \gamma U + (1-\gamma)B$ with $\gamma \leq \frac{1}{2}$; 
 	\item Among the initial $k$ terms $\set{a_n}_{n=1}^k$, less than $	c k (U-B)\min\set{\gamma, \varepsilon_k\delta_k}$ terms lie outside the interval $[B,U]$. \end{enumerate}
 Then for all $n > k$,
 \begin{equation*}
 	a_n \leq B+ (U-B)\begin{cases}
			C \,{\gamma}/{\varepsilon_k \delta_k} &\text{ for } \gamma \leq 1/2, \\
			1 - c\,{\varepsilon_k \delta_k}/{\gamma} & \text{ for } \gamma \geq \varepsilon_k\delta_k /2C
		\end{cases}
 \end{equation*}
 
\end{prop}

\Cref{prop:technical-convergence} will be used to show that intervals capturing the tails of the sequence contract sufficiently to ensure convergence.

In contrast, \cref{prop:technical-divergence} shows that if the initialization of the sequence contains a sufficiently large proportion of terms exceeding a threshold $U$, then these values can force subsequent iterates to increase.

\begin{prop}\label{prop:technical-divergence}
Suppose $\varepsilon_n = A (\log n)^{-\alpha}$ and $\delta_n = B (\log n)^{-\beta}$ with $\alpha,\beta > 0$. There exists a constant $C$ such that the following holds. Suppose $k \geq C$  and 
let $\set{a_n}_{n=1}^\infty$ be a sequence with initialization $a_1, \ldots, a_k \in [0,1]$  satisfying 
	\begin{equation*}
		a_{n+1} = \varepsilon_n \bar{a}_n + (1-\varepsilon_n)\bar{a}^{(n)}_{\delta_n n} \text{ for all } n \geq k.
	\end{equation*} 
	Suppose that  among the initial $k$ terms $\set{a_n}_{n=1}^k$, at least $\delta_k^{1/2}$ fraction exceed $U$. Then, among the first $k  / 2\delta_k^{1/2}$ terms, at least half satisfy
	\begin{equation*}
		a_n \geq \paren{1- C\varepsilon_k \delta_k^{1/2}}U.
	\end{equation*}
\end{prop}

\Cref{prop:technical-divergence} will be used to show that the sequence can be driven sufficiently "back and forth" to prevent convergence.

\begin{rmk}
 \Cref{prop:technical-convergence} is a partial converse of \cref{prop:technical-divergence}. In particular, note that if  $B = 0$ and $\gamma=  \delta_k^{1/2} $, then  it promises $a_n \leq U(1 - c\varepsilon_k \delta_k^{1/2})$, while from the latter we have $a_n \geq U(1-C \varepsilon_k \delta_k^{1/2})$. This is where the threshold $\alpha + \beta/2 = 1$ arises. 
\end{rmk}

\subsection{Tools for Upper Bounds}
In this section, we prove a helper lemma, which we will later lift by majorization to prove \cref{prop:technical-convergence}. The latter is the main tool we will use for upper bounds.

\begin{lemma}\label{lem:technical-helper}
	Suppose $\varepsilon_n = A(\log n)^{-\alpha}$ and $\delta_n = B(\log n)^{-\beta}$ with $\alpha, \beta \geq 0$. There exists a constant $C$ such that the following holds. Suppose $k \geq C$  and let $\set{a_n}_{n=1}^\infty$ be a sequence with initialization $a_1, \ldots, a_k \in [0,1]$  satisfying
\begin{equation*}
		a_{n+1} = \varepsilon_n \bar{a}_n + (1-\varepsilon_n)\bar{a}^{(n)}_{\delta_n n} \text{ for all } n \geq k.
\end{equation*} 
Suppose $\bar{a}_k\leq \gamma$. Then, for all $n > k$ \begin{equation*}
	a_n \leq C \frac{\gamma }{\varepsilon_k \delta_k}.
\end{equation*} 
\end{lemma}
	
\begin{proof}  
	By \cref{lem:majorization}, we assume without loss of generality that the initialization $\set{a_n}_{n=1}^k$ is comprised of $\gamma k$ ones and $(1-\gamma)k$ zeroes, as this sequence majorizes all others in $[0,1]$ with the same mean. Since the recursion is monotone with respect to majorization, any upper bound proved for this dominating sequence automatically applies to the original sequence.

	Set $C$ to be a large constant to be specified in the argument below. Define a sequence $\set{u_n}_{n = 1}^\infty$ by setting $u_1 = \cdots = u_k = 0$ and letting 	\begin{equation*}
		u_{n+1} = \varepsilon_n \bar{u}_n + (1-\varepsilon_n)u_n + \frac{\gamma k}{\delta_n n} \text{ for all } n \geq k.
		\end{equation*}

	\begin{claim}\label{claim:increasing-u}
		The sequence $\set{u_n}_{n=1}^\infty$ is non-decreasing and 
		\begin{equation*}
			\lim_{n \to \infty} u_n \leq  C\frac{\gamma}{\varepsilon_k\delta_k}
		\end{equation*}
	\end{claim}
	Assuming \cref{claim:increasing-u}, to finish the proof of \cref{lem:technical-helper}, it suffices to prove  $a_n \leq u_n$ for all $n > k$. This follows from an inductive argument. In what follows, we verify the base case and the inductive step simultaneously. Suppose that $a_m \leq u_m$ for all $k < m \leq n$. Then,
	\begin{equation*}
		\bar{a}_n  = \frac{k\bar{a}_k + a_{k+1} + \cdots + a_n}{n} \leq \frac{k\bar{a}_k + u_{k+1} + \cdots + u_n}{n} =\frac{n \bar{u}_n + \gamma k}{n}.
	\end{equation*} 
	Observe this bound also holds for the base case $n = k$. Continuing,
	\begin{equation*}
		\bar{a}_{\delta_n n}^{(n)}=\frac{\gamma k  +a_{\gamma k + 1}^{(n)} + \cdots + a_{\delta_n n}^{(n)}}{\delta_n n} \leq \frac{\gamma k + (\delta_n n - \gamma k) u_n}{\delta_n n}.
	\end{equation*}	
	This bound also holds for the base case $n = k$. Combining the two bound, we get
		\begin{align*}
		a_{n+1} &= \varepsilon_n \bar{a}_n + (1-\varepsilon_n)\bar{a}_{\delta_n n}^{(n)}  \\
				&\leq \varepsilon_n \frac{n \bar{u}_n + \gamma k}{n} + (1-\varepsilon_n) \frac{(\delta_n n -\gamma k)u_n + \gamma k}{\delta_n n} \\
				&= \varepsilon_n\bar{u}_n + (1-\varepsilon_n)u_n + \frac{(1-\varepsilon_n + \varepsilon_n \delta_n) \gamma k}{\delta_n n} \leq u_{n+1},
	\end{align*}
	which proves $a_{n+1} \leq u_{n+1}$ and hence the \cref{lem:technical-helper}. Now it remains to prove \cref{claim:increasing-u}.
	
	\begin{proof}[Proof of \Cref{claim:increasing-u}]
	Define $\Delta_n \triangleq  n(u_n - \bar{u}_n)$. Then, the recurrence is rewritten as
	\begin{align*}
		\Delta_{n+1} = (n+1)(u_{n+1} &- \bar u_{n+1}) = n u_{n+1} - n\bar{u}_n \\
					 &= \varepsilon_nn\bar{u}_n + (1-\varepsilon_n)nu_n + \frac{\gamma k}{\delta_n} - n\bar {u}_n=(1-\varepsilon_n)\Delta_n + \frac{\gamma k}{\delta_n}.
	\end{align*}
	We first show that this recursive formula implies that
	\begin{equation}\label{eq:Delta-bound}
		\Delta_n \leq \frac{\gamma k}{\varepsilon_n \delta_n}.
	\end{equation}
	We prove this bound by induction. The base case $j \leq k$ follows from $\Delta_j = 0$. Now 
	\begin{equation*}
		\Delta_{n+1} = (1-\varepsilon_{n})\Delta_n + \frac{\gamma k}{\delta_n}\leq (1-\varepsilon_n)\frac{\gamma k}{\varepsilon_n\delta_n}  + \frac{\gamma k}{\delta_n}= \frac{\gamma k}{\varepsilon_n \delta_n} \leq \frac{\gamma k}{\varepsilon_{n+1}\delta_{n+1}}.
	\end{equation*}
	Here we used $\varepsilon_{n+1}\delta_{n+1} \leq \varepsilon_n \delta_n$, which holds as long as $C$ is a sufficiently large constant. We now show that $\set{\Delta_n}_{n = k+1}^\infty$ is non-decreasing.  This follows immediately from the identity
	\begin{equation}\label{eq:Delta-increasing}
		\Delta_{n+1} = \Delta_n - \varepsilon _n\Delta_n + \frac{\gamma k}{\delta_n} \geq \Delta_n.
	\end{equation} 
	Now we translate (\ref{eq:Delta-bound}) and (\ref{eq:Delta-increasing}) to bounds on the sequence $u_n$ using the identity 
	\begin{equation*}
		\frac{\Delta_{n+1} - \Delta_n}{n} = \frac{(n+1)u_{n+1} - (n+1)\bar{u}_{n+1} -nu_n + n\bar{u}_n}{n} = u_{n+1} - u_n.
	\end{equation*}
	Moreover, (\ref{eq:Delta-increasing}) immediately implies that $\set{u_n}_{n =1}^\infty$ is non-decreasing. Now
	\begin{equation*}
		u_{n+1} = u_n + \frac{\Delta_{n+1} - \Delta_n}{n} \implies u_{n+1} = \frac{\Delta_{n+1}}{n} + \sum_{t = k+1}^{n}\frac{\Delta_t}{t(t-1)}
	\end{equation*}
	Applying the bound (\ref{eq:Delta-bound}) to this, we get  
	\begin{equation*}
		u_{n+1} \leq \frac{\gamma k}{\varepsilon_{n+1}\delta_{n+1} n} + \sum_{t = k+1}^\infty \frac{\gamma k}{\varepsilon_t \delta_t t (t-1)}.
	\end{equation*}
	If $C$ is a sufficiently large constant, we have $\varepsilon_{n+1}\delta_{n+1}n \geq \varepsilon_k \delta_k  k$ for all $n \geq k$, so the first term is bounded by $\gamma /\varepsilon_k \delta_k$. We bound the second term by splitting the sum into two parts 
		\begin{equation*}
		\sum_{t = k+1}^\infty\frac{1}{\varepsilon_t\delta_t t(t-1)} \leq 
		\sum_{t = k+1}^{k^2}\frac{1}{\varepsilon_t\delta_t t(t-1)} + \sum_{t = k^2}^{\infty}\frac{1}{\varepsilon_t\delta_t t(t-1)}.
	\end{equation*}
	Since $\varepsilon_t$ and $\delta_t$ decay polylogarithmically,  there exists a small constant $c$ such that   	
	 $\varepsilon_t \delta_t \geq c\varepsilon_k \delta_k$ for $t \leq k^2$ and    $\varepsilon_t \delta_t t (t-1) \geq ct^{3/2}$ for $t\geq k^2$. Hence,
	\begin{align*}
		\sum_{t = k+1}^{k^2}\frac{1}{\varepsilon_t\delta_t t(t-1)} + \sum_{t = k^2}^{\infty}\frac{1}{\varepsilon_t\delta_t t(t-1)} &\leq \frac{1}{c\varepsilon_k\delta_k}\sum_{t = k}^\infty \frac{1}{t^2} + \frac{1}{c}\sum_{t = k^2}^\infty \frac{1}{t^{3/2}} \leq (C-1)\frac{1}{\varepsilon_k\delta_k k}.
	\end{align*}
	In the last inequality, we chose $C$ to be a large enough constant. Substituting this into the bound for $u_{n+1}$ above, we conclude the proof.
	
	 	\end{proof}
	Having proved \cref{claim:increasing-u}, we conclude the proof of \cref{lem:technical-helper}.
	\end{proof}

\begin{lemma}\label{lem:technical-convergence-helper}
Suppose $\varepsilon_n = A(\log n)^{-\alpha}$ and $\delta_n = B(\log n)^{-\beta}$ with $\alpha, \beta \geq 0$. There exist constants $c$ and $C$ such that the following holds. Suppose $k \geq C$  and let $\set{a_n}_{n=1}^\infty$ be a sequence with initialization $a_1, \ldots, a_k \in [0,1]$  satisfying
\begin{equation*}
		a_{n+1} = \varepsilon_n \bar{a}_n + (1-\varepsilon_n)\bar{a}^{(n)}_{\delta_n n} \text{ for all } n \geq k.
\end{equation*} 
Suppose $\bar{a}_k\leq \gamma \leq \frac{3}{4}$. Then, for all $n > k$ we have 
\begin{equation*}
		a_n \leq 
		\begin{cases}
			C \,{\gamma}/{\varepsilon_k \delta_k} &\text{ for } \gamma \leq 3/4, \\
			1 - c\,{\varepsilon_k \delta_k}/{\gamma} & \text{ for } \gamma \geq \varepsilon_k\delta_k/16C
		\end{cases}
	\end{equation*} 
\end{lemma}
	
\begin{proof}Let $K$ be the constant from \cref{lem:technical-helper}. Fix $c$ and $C$ to be some small and large constant, respectively. 	Since $k \geq C$ and $\varepsilon_n$ and $\delta_n$ satisfy logarithmic decay, we can take $C$ large enough so that 
	 \begin{enumerate}
	 	\item both $\set{\varepsilon_n}_{n=k}^\infty$ and $\set{\delta_n}_{n=k}^\infty$ are non-increasing and at most $\frac{1}{2}$;
	 	\item the sequence $\set{\varepsilon_n \delta_n n}_{n=k}^\infty$ is increasing;
	 	\item if $m$ is chosen so that $\varepsilon_m\delta_mm \triangleq  K k $ then 
	 \begin{equation*}
		\varepsilon_m \geq \frac{\varepsilon_k}{K} \text{ and } \delta_m \geq \frac{\delta_k}{K}.
	\end{equation*}
	 \end{enumerate}
	 We now begin the proof.  \Cref{lem:technical-helper} shows that for all $n > k$ we have $a_n \leq K \gamma /(\varepsilon_k \delta_k)$. Taking $C \geq K$, we obtain the first part of the desired bound. In what follows, we assume $\gamma \geq \varepsilon_k\delta_k/16C$. First, we use induction to show that  for all $k < n \leq m$ 
	\begin{equation*}
		a_n \leq \lambda \triangleq 1 - \frac{\varepsilon_m \delta_m  (1-\gamma)}{\varepsilon_m \delta_m + (1-\varepsilon_m)\gamma}.
	\end{equation*}
	We verify the base case and the inductive step simultaneously. Assume $a_j \leq \lambda$ for all $k <j \leq n$. Since $n \leq m$, we get
	\begin{align*}
		a_{n+1} &= \varepsilon_n \bar{a}_n + (1-\varepsilon_n)\bar a_{\delta_n n}^{(n)} \\
				&\leq \varepsilon_m \bar{a}_n +(1-\varepsilon_m)\bar a_{\delta_m n}^{(n)}  \\
				&\leq \varepsilon_m\frac{\gamma k + \lambda(n-k)}{n} + (1-\varepsilon_m)\frac{\gamma k + \lambda(\delta_m n - \gamma k)}{\delta_m n} \leq \lambda. 
	\end{align*}
	Observe that this bound also holds for the base case $n = k$, so the induction is complete. Next, we recall that by assumption  $\varepsilon_k\delta_k/16C <\gamma \leq 3/4$, so
		\begin{equation*}
		\frac{\varepsilon_m \delta_m  (1-\gamma)}{\varepsilon_m \delta_m + (1-\varepsilon_m)\gamma} \geq \frac{1}{4}\frac{\varepsilon_m \delta_m}{\varepsilon_m\delta_m + \gamma} \geq  \frac{1}{4 (16 C + 1) }\frac{\varepsilon_m\delta_m}{\gamma  } \geq \frac{\varepsilon_m \delta_m}{100C \gamma }
	\end{equation*}
	Since our goal is to obtain an upper bound on the terms of the sequence, it suffices to establish the bound for a simpler sequence that dominates the original one. More precisely, by applying  \cref{lem:majorization} we may assume without loss of generality that among $\set{a_n}_{n=1}^m$ there are $\gamma k$ ones, and the remaining terms are equal to $\lambda_+$ where  
		\begin{equation*}
		\lambda_+ \triangleq  1- \frac{\varepsilon_m\delta_m}{100C\gamma} \geq  1 - \frac{\varepsilon_m \delta_m(1-\gamma)}{\varepsilon_m \delta_m + (1-\varepsilon_m)\gamma } = \lambda.
	\end{equation*}
	By affine invariance of the recurrence, we may shift and rescale the sequence so that $\lambda_+$ is sent to $0$ and $1$ is sent to $1$. Applying \cref{lem:technical-helper} in this scale and then undoing the affine transformation yields the desired bound for the original sequence. More precisely, we conclude that for all $n > k$ 	
	\begin{equation*}
		a_n \leq \lambda_++(1-\lambda_+)K\frac{\gamma k}{\varepsilon_m\delta_m m} \leq \lambda_+ + (1-\lambda_+) \times \frac{3}{4} = 1-\frac{\varepsilon_m\delta_m}{400C\gamma} \leq 1-c\frac{\varepsilon_k \delta_k}{\gamma}.
	\end{equation*}
	In the last step, we chose $c = 1/(400CK^2)$.
	\end{proof}
	
We are now ready to prove \cref{prop:technical-convergence}.

\begin{proof}[Proof of \cref{prop:technical-convergence}]

	We will apply majorization lemma (\cref{lem:majorization}) and reverse majorization lemma (\cref{lem:reverse-majorization}) to reduce to \cref{lem:technical-convergence-helper}.
	Since our goal is to obtain an upper bound on the terms of the sequence, it suffices to establish the bound for a simpler sequence that dominates the original one. To this end, we apply \cref{lem:majorization} to assume without loss of generality that all terms originally lying in the interval $[B,U]$ are moved to one of the endpoints $\set{B,U}$ in a way that preserves both the average and the number of terms lying outside the interval.
	
	We may apply \cref{lem:majorization} once again to replace all terms not in $\set{B,U}$ by $1$. At this point, every term in the initialization $\set{a_n}_{n=1}^k$  takes one of the values $\set{B,U,1}$. Note that this step alters the average.  We will quantify this change and show that it affects the situation only marginally.	
	
	Denote by $k_B$ the number of terms equal to $B$, and similarly define $k_U$ and $k_1$. Let $U_+$ be the average of the terms are equal to $U$ or $1$, and let $\gamma_+$ be the fraction of terms taking values in $\set{U,1}$. Thus,	
	\begin{equation*}
	     U_+ \triangleq \frac{k_UU + k_1}{k_U + k_1} \text{ and }\gamma_+ \triangleq  \frac{k_U+k_1}{k}.
	\end{equation*}
	We first show that $U_+ \approx U$ and $\gamma_+ \approx \gamma$. 
	\begin{claim}\label{claim:helper}
		With the notation above, we have 
		\begin{equation*}
		  (1-c)\gamma \leq \gamma_+ \leq (1 + c)\gamma \text{ and } 	U \leq U_+ \leq U + 2c(U-B)\min\set{1,\frac{\varepsilon_k \delta_k}{\gamma}}
		\end{equation*}
	\end{claim}
	\begin{proof}
	Since the number of terms outside the interval $[B,U]$ has not changed, we have $k_1 \leq c\gamma k(U-B)$. Moreover, by replacing at most $k_1$ terms by $1$, the total sum can increase by at most $k_1$. Therefore,
		\begin{align*}
			\bar{a}_k \leq 	\gamma U + (1-\gamma)B + \frac{k_1}{k} \leq (1 + c) \gamma U + (1-(1+c)\gamma)B. 
		\end{align*}
		By Markov inequality, this implies the upper bound $\gamma_+ \leq (1+c)\gamma$.  
		
		\medskip
		
	 	For the lower bound on $\gamma_+$, note that replacing some terms by $1$ can only increase the average. Hence,
	 	\begin{align*}
	 		Bk + \gamma k (U-B) \leq k\bar{a}_k &= Uk_U + Bk_B + k_1 \\
	 						   				  &\leq k_U(U-B) +Bk +k_1 \\
	 						   				  &\leq k_U(U-B) + Bk + c\gamma k(U-B)
	 	\end{align*}
	 	In particular, this implies $\gamma_+ \geq \frac{k_U}{k} \geq (1-c)\gamma$. Finally, observe that lower bound on $U_+$ is trivial, while upper bound follows from
	 	\begin{align*}
	 		U_+ \leq \frac{Uk_U + k_1}{k_U}  &\leq U + \frac{c k(U-B)}{(1-c)\gamma k}\min\set{\gamma, \varepsilon_k\delta_k} \\
	 		&\leq U + 2c(U-B)\min\set{1,\frac{\varepsilon_k \delta_k} {\gamma}}.
	 	\end{align*}
	 	Here we chose $c < 1/2$. 
	\end{proof}

	Let us now collect the all terms taking value in $\set{U,1}$ into their average $U_+$. By \cref{lem:reverse-majorization}, this operation does not affect the future evolution of the sequence, provided that the newly generated terms remain bounded by $U$. In particular, as long as we establish the bound $a_n \leq U$ for all $n > k$, which is weaker than the conclusion of the proposition, we may apply reverse majorization without loss of generality.	\medskip
	
	In this transformed sequence, we have $a_1, \ldots, a_k \in \set{B, U_+}$, with the fraction of terms equal to $U_+$ being exactly $\gamma_+$. By \cref{claim:helper}, we have $\gamma_+\leq (1+c)\gamma \leq 3/4$.
	
	In particular, after shifting and rescaling the sequence to take values in $\set{0,1}$, we may apply \cref{lem:technical-convergence-helper} and undo the transformation to conclude that there exist constants $c'$ and $C'$, such that for all $n > k$, 
	\begin{align*}
		a_n & \leq 
		B + (U_+ - B)\begin{cases}
			C' \gamma_+ / \varepsilon_k \delta_k  & \text{ for } \gamma_+ \leq 3/4,\\
			1-c'\varepsilon_k \delta_k / \gamma_+ & \text{ for } \gamma_+ \geq \varepsilon_k \delta_k /16C'
		\end{cases}
	\end{align*}	
	It now suffices to bound the right-hand side, for which we invoke \cref{claim:helper} again. First, whenever $\gamma \leq 1/2$, we have $\gamma_+ \leq 3/4$. Hence, for  $C=4C'$ we obtain 
		\begin{align*}
		a_n \leq  B+ C' \frac{\gamma_+}{\varepsilon_k \delta_k}(U_+ - B) \leq B + 4C'\frac{\gamma }{\varepsilon_k\delta_k}(U-B) = B+C\frac{\gamma}{\varepsilon_k\delta_k}(U-B).
	\end{align*}
	On the other hand, whenever $\gamma \geq \varepsilon_k \delta_k/2C$, we have $\gamma_+ \geq \varepsilon_k \delta_k/16C'$, so we obtain
	\begin{align*}
		a_n &\leq B + (U_+- B)\paren{1-\frac{c' \varepsilon_k\delta_k}{\gamma}} \\
		&\leq B + (U - B)\paren{1+\frac{2c\varepsilon_k\delta_k}{\gamma}}\paren{1-\frac{c' \varepsilon_k\delta_k}{\gamma}} \\
		&\leq B + (U-B)\paren{1-\frac{c\varepsilon_k\delta_k}{\gamma}}.	\end{align*}
	Here we applied $(U_+-U)\leq 2\,c\,{\varepsilon_k\delta_k}(U-B)/\gamma $ and chose $c$ to be a sufficiently small constant relative to $c'$. \end{proof}

\subsection{Tools for Lower Bounds}

In this section, we first establish a helper lemma, which we then lift via majorization to prove \cref{prop:technical-divergence}. The latter is the main tool for lower bounds.

\begin{lemma}\label{lem:technical-helper-2}
	Fix $0 <\varepsilon, \delta, \gamma \leq \frac{1}{2}$ such that $\gamma \geq 2\delta$.  Suppose $\set{a_n}_{n=1}^\infty$ is a sequence with initialization  $a_1 = \cdots = a_{\gamma k} = 1$ and $a_{\gamma k+1} = \cdots = a_{k} = 0$, satisfying
	\begin{equation*}
		a_{n+1} \geq  \varepsilon \bar{a}_n + (1-\varepsilon)\bar{a}^{(n)}_{\delta n} \text{ for all } n\geq k.
	\end{equation*} 
	Then, at least half of the terms up to index	 $\gamma k/2\delta $ satisfy 
	\begin{equation*}
		a_n \geq 1-\varepsilon \paren{\frac{5\delta}{\gamma}}^{1-\varepsilon}.
	\end{equation*}
\end{lemma}
\begin{proof}
	Let $n_0 \triangleq \gamma k / (4\delta)$.  Observe that if $n \leq 2n_0$ then $\delta n \leq 2\delta n_0 \leq \gamma k $. In particular $\bar{a}_{\delta n}^{(n)} = 1$. Thus, as long as $n \leq 2n_0$, we have 
	\begin{equation*}
		a_{n+1} \geq   \varepsilon \bar{a}_n + (1-\varepsilon) \quad \text{ so }\quad   \bar{a}_{n+1} \geq  \frac{(n+\varepsilon)\bar{a}_n + (1-\varepsilon)}{n+1}.
	\end{equation*}
	We can solve this recurrence by setting $\bar{a}_n = 1-d_n$ and observing that 
	\begin{align*}
		d_{n+1} \leq d_n \paren{1 - \frac{1-\varepsilon}{n+1}} \quad \text{ so } \quad d_n &\leq  (1-\gamma)\prod_{j = k+1}^n\paren{1- \frac{1-\varepsilon}{j}} \\
		& \leq \exp\set{(1-\varepsilon)\log\paren{\frac{k+1}{n+1}}}.
	\end{align*}
	Since $\bar{a}^{(n)}_{\delta n} \geq \bar{a}_n$, the sequence $\bar{a}_n$ is non-decreasing. Hence, for all $n \geq n_0$ 
	\begin{equation*}
		\bar{a}_n \geq 1-\paren{\frac{k+1}{n_0+1}}^{1-\varepsilon} \geq 1- \paren{\frac{5\delta }{\gamma }}^{1-\varepsilon}.
	\end{equation*} 
	This means that whenever $n_0 \leq n <2n_0$ we have that 
	\begin{equation*}
		a_{n+1} \geq \varepsilon \bar{a}_n + 1 -\varepsilon \geq 1-\varepsilon \paren{\frac{5\delta }{\gamma}}^{1-\varepsilon}.
	\end{equation*}
	This concludes the proof.
\end{proof}

We are now ready to prove \cref{prop:technical-divergence}.

\begin{proof}[Proof of \cref{prop:technical-divergence}]
Fix a constant $C$ to be specified below. For notational convenience, set $\varepsilon \triangleq \varepsilon_k$, $\delta \triangleq  \delta_k$, and $\gamma \triangleq  \delta_k^{1/2}$. 

By the majorization lemma (\cref{lem:majorization}), we may assume without loss of generality that among the initial terms $\set{a_n}_{n = 1}^k$, a $\gamma$ fraction are equal to $U$, while the remaining terms are equal to $0$. 

If $C$ is chosen sufficiently large, then $0 \leq \varepsilon, \delta, \gamma \leq 1/2$ and $\gamma \geq 2\delta$. Moreover, we may assume that $\varepsilon_n$ and $\delta_n$ are non-increasing for $n \geq k$. In this case, we may rescale the sequence so that all terms take values in $\set{0,1}$, and apply \cref{lem:technical-helper-2} to conclude that, for at least half the terms up to $ k/2\delta^{1/2}$ satisfy
\begin{equation*}
	a_n \geq \sqb{1 - \varepsilon \paren{\frac{5\delta}{\gamma}}^{1-\varepsilon}}U 
\end{equation*}
It now suffices to show that $(5\delta / \gamma)^{1-\varepsilon} \leq C\delta / \gamma$ provided that $C$ is chosen sufficiently large. This follows from the observation that
\begin{equation*}
	\log \set{ \paren{\frac{\gamma}{\delta}}^{\varepsilon} } = \frac{\varepsilon}2 \log \frac{1}{\delta } = A \frac{\beta \log \log k -\log B }{2 (\log k)^\alpha}  = O(1).
\end{equation*}
\end{proof}

\section{Proof of the Main Theorem}

We first apply \cref{prop:technical-convergence} to prove convergence in \cref{thm:convergence}. 

\begin{theorem}\label{thm:convergence}
	Let $A,B > 0$ and $\alpha, \beta \geq 0$ be constants satisfying $\alpha + \beta / 2 \leq 1$. 
	Suppose $\varepsilon_n = A (\log n)^{-\alpha}$ and $\delta_n = B(\log n)^{-\beta}$ and let $\set{a_n}_{n=1}^\infty$ be a sequence with initialization $(a_1, \ldots, a_k)$ that satisfies   
	\begin{equation*}
		a_{n+1} \in \varepsilon_n \bar{a}_n + (1-\varepsilon_n)[\bar a_{\delta_n n}^{[n]}, \bar a_{\delta_n n}^{(n)}] \text{ for all } n \geq k.
	\end{equation*} 
	Then the sequence necessarily converges.  
\end{theorem}

\begin{proof}
Since the recurrence is affine invariant, we may shift and rescale the initialization to assume without loss of generality that $a_1, \ldots, a_k \in [0,1]$.  Furthermore, we may assume that $k$ arbitrarily large.

	We will construct a nested sequence of shrinking intervals $[B_T,U_T]$ that capture the tails of the sequence. The intervals will be defined inductively in stages. More precisely, at each stage $T \in \NN_0$, we will maintain an interval $[B_T,U_T]$, an index  $n_T$, and the guarantee that $a_n \in [B_T, U_T]$ for all $n \geq n_T$.
	
	At stage $T = 0$, we initialize with $B_0 = 0$, $U_0 = 1$, and $n_0 = k$. Since every  term of the sequence is a convex combination of $a_1, \ldots, a_k$, all subsequent terms remain bounded in $[0,1]$. In particular, $a_n \in [0,1]$ for all $n \geq k$, so the base case holds.
	
	Assume now that we are at stage $T$. We define the starting index of the next stage implicitly by setting 
	\begin{equation*}
		n_{T+1} \triangleq  n_T \times \frac{K^2}{\varepsilon_{n_{T+1}}^2 \delta_{n_{T+1}}^3(U_T - B_T)^2}.
	\end{equation*}
	Here $K$ is a sufficiently large constant to be specified below. This definition is well posed as the sequence  $m \varepsilon_m^2 \delta_m^3$ is   increasing for $m \geq k$ when $k$ is sufficiently large.  For notational convenience, we set $\varepsilon_* \triangleq \varepsilon_{n_{T+1}}$ and $\delta_* \triangleq \delta_{n_{T+1}}$.

	Below, we construct the next interval $ [B_{T+1}, U_{T+1}] \subset [B_T, U_T]$ such that $ a_{n} \in [B_{T+1}, U_{T+1}]$ for all $n \geq n_{T+1}$, and 
	\begin{equation*}
		(U_{T+1}- B_{T+1}) \leq (U_T - B_T) \times (1 - c \varepsilon_{*} \delta_{*}^{1/2})
	\end{equation*}
	for a sufficiently small constant $c$, independent of $T$. 
	
	Observe that this suffices to complete the proof, for the following reason. Given the contraction bound above, \cref{prop:convergence-contraction} implies that $(U_T - B_T) \to 0$ as $T \to \infty$. Since the intervals are nested, there exists a limit $L$ such that $B_T \to L$ and $U_T \to L$. Moreover, since $a_n \in [B_T, U_T]$ for all $n \geq n_T$, we conclude that 
	\begin{equation*}
		\liminf_{n \to \infty}a_n  = L = \limsup_{n \to \infty}a_n.
	\end{equation*}
	This implies that the sequence converges.

	\medskip
	It now suffices to construct the interval with the desired properties. 
	For the purposes of the analysis, we define an intermediate index by
	\begin{equation*}
		n_T^+ \triangleq n_T \times \frac{K}{\varepsilon_{*}\delta_{*}^2(U_T-B_T)}.
	\end{equation*}
	To aid the reader’s understanding, let us outline the reasoning behind the argument.
	
	\begin{itemize}
\item[-]
Recall that, by inductive hypothesis, $a_n \in [B_T, U_T]$ for all $n \geq n_T$. This means that if we wait long enough, an overwhelming majority of the sequence will lie in $[B_T, U_T]$. This allows us to effectively zoom in on this interval and apply our technical lemma. More precisely, by time $n_T^+$ the second condition of \cref{prop:technical-convergence} is met.

\item[-]
If, at any time $m$ between $n_T^+$ and $n_{T+1}$, the average $\bar a_m$ comes too close to $B_T$, then by \cref{prop:technical-convergence}, subsequent terms will be bounded away from $U_T$. This lets us shrink the interval. By symmetry, the same argument applies with the roles of $U_T$  and $B_T$ reversed.

\item[-]
Alternatively, if at no time between $n_T^+$ and $n_{T+1}$ does the average come close to $B_T$ or $U_T$, then the sequence remained away from both endpoints for too long. In this case, by \cref{prop:technical-convergence} again, all subsequent terms remain uniformly bounded away from at least one endpoint, which again allows us to shrink the interval.
\end{itemize}
	
	We now carry out the above analysis as follows. First, suppose that there exists an index $m$ with $n_T^+ \leq m \leq n_{T+1}$, such that $\bar{a}_m \leq \delta_{m}^{{1/2}}U_T + (1-\delta_m^{1/2})B_T$. Since all terms $\set{a_j : n_T \leq j \leq m}$ are contained in $[B_T, U_T]$, the fraction of terms lying outside the interval is at most 
	\begin{equation*}
		\frac{n_T}{m} \leq \frac{n_T}{n_T^+} = \frac{\varepsilon_{*} \delta_{*}^2(U_T-B_T)}{K} \leq \frac{\varepsilon_m \delta_m (U_T - B_T)}{K}.
	\end{equation*}
	Therefore, if $K$ is chosen sufficiently large, the conditions \cref{prop:technical-convergence} are met. It follows that for sufficiently small constant $c > 0$ and all $n \geq m$, 
	\begin{align*}
		a_n &\leq U_T - c\,\frac{\varepsilon_m \delta_m}{\delta_m^{1/2}}(U_T - B_T) \leq U_T - c\,\varepsilon_{*}\delta_{*}^{{1/2}}(U_T- B_T).
	\end{align*}
	In this case, we may set $B_{T+1} \triangleq  B_T$ and $U_{T+1} \triangleq  U_T - c\,\varepsilon_{*}\delta_{*}^{1/2}(U_T- B_T)$ to finish the proof.	
	\medskip
	
	Analogously, if  there exists $m\in [n_T^+, n_{T+1}]$ such that $\bar{a}_m \geq \delta_{m}^{{1/2}}B_T + (1-\delta_m^{1/2})U_T$, then we may instead set $U_{T+1} \triangleq U_T$ and $B_{T+1} \triangleq B_T +c\varepsilon_* \delta_*^{1/2} (U_T - B_T)$ to finish the proof.	
	\medskip
	
	In the remainder of the proof, we assume that for all $m \in [n_T^+, n_{T+1}]$,  
	\begin{equation*}
		\delta_{m}^{{1/2}}U_T + (1-\delta_m^{1/2})B_T < \bar{a}_m < \delta_{m}^{{1/2}}B_T + (1-\delta_m^{1/2})U_T.
	\end{equation*}
	Define the intermediate candidate endpoints by 
	\begin{align*}
		U_{T}^- \triangleq U_T - \frac{1}{2}\varepsilon_* \delta_*^{{1/2}}(U_T - B_T) \quad \text{ and } \quad 
		B_{T}^+ \triangleq B_T + \frac{1}{2}\varepsilon_* \delta_*^{{1/2}}(U_T - B_T).
	\end{align*}

	First, we will show that $a_{m+1} \in [B_{T}^+, U_{T}^-]$ for all $m \in [n_T^+, n_{T+1})$. To this end, observe that
		\begin{align*}
		a_{m+1} &\leq \varepsilon_m \bar{a}_m + (1-\varepsilon_m)\bar{a}_{\delta_m m }^{(m)} \\
		&\leq \varepsilon_* \bar a_m + (1-\varepsilon_*)\bar{a}_{\delta_* m }^{(m)}\\
			&\leq \varepsilon_*[\delta_*^{{1/2}}B_T + (1-\delta_*^{1/2})U_T] + (1-\varepsilon_*)\sqb{\frac{1 \times n_T + U_T \times (\delta_* m - n_T)}{\delta_* m}}\\
		&= \varepsilon_*\delta_*^{{1/2}}B_T + (1-\varepsilon_*\delta_*^{{1/2}})U_T + (1-U_T)(1-\varepsilon_*)\frac{n_T}{\delta_* m}.
	\end{align*}
	The second inequality above follows from monotonicity $\delta_* \leq \delta_m$. In the third inequality, we use the crude bounds: $a_n \leq 1$ for $n \leq n_T$ and $a_n \leq U_T$ if $a_n > n_T$. To conclude, we bound the final term as
	\begin{align*}
	(1-U_T)(1-\varepsilon_*)\frac{n_T}{\delta_* m} &\leq \frac{n_T}{\delta_* n_T^+} = \frac{1}{K}\varepsilon_* \delta_* (U_T - B_T) \leq \frac{1}{2}\varepsilon_*\delta_*^{1/2}(U_T - B_T). 
		\end{align*}
	Substituting this bound above yields $a_{m+1} \leq U_{T}^-$ and analogously $a_{m +1} \geq B_T^+$.

	Define $\gamma$ by $\bar a_{n_{T+1}} = \gamma U_T^- + (1-\gamma) B_T^+$. Without loss of generality, let us assume that $\gamma \leq 1/2$. Recall that $\bar a_{n_{T+1}} \geq \delta_*^{1/2}U_T + (1-\delta_*^{1/2})B_T$.	Assuming $\varepsilon_*$ is sufficiently small, this implies  $\gamma \geq 2\delta_*^{1/2}$.
	
	Furthermore, since $a_{m+1} \in [B_T^+, U_T^-]$ for all $m \in [n_T^+, n_{T+1})$, the fraction of terms up to time $n_{T+1}$ that fall outside the interval $[B_{T}^+, U_{T}^-]$ is at most 
	\begin{equation*}
		\frac{n_T^+}{n_{T+1}} = \frac{\varepsilon_*\delta_*(U_T - B_T) }{K} \leq  \frac{2\,\varepsilon_*\delta_*(U_{T}^- - B_{T}^+)}{K}.
	\end{equation*}
	Hence, by  \cref{prop:technical-convergence} we get $a_n \leq U_T^-$ for all $n \geq n_{T+1}$.	In particular, we may set $B_{T+1} = B_T$ and $U_{T+1} = U_T^-$ to conclude the proof.
	 \end{proof}

We now apply \cref{prop:technical-divergence} to prove the non-convergence in \cref{thm:divergence}.
	
\begin{theorem}\label{thm:divergence}
Let $A,B > 0$ and $ \alpha,\beta > 0$ be constants satisfying $\alpha + \beta / 2 > 1$. 
	Suppose $\varepsilon_n = A (\log n)^{-\alpha}$ and $\delta_n = B(\log n)^{-\beta}$ . There exists a sequence $\set{a_n}_{n=1}^\infty$  and an initialization $(a_1, \ldots, a_k)$ satisfying
	\begin{equation*}
		a_{n+1} \in \varepsilon_n \bar{a}_n + (1-\varepsilon_n)[\bar a_{\delta_n n}^{[n]}, \bar a_{\delta_n n}^{(n)}] \text{ for all } n \geq k
	\end{equation*} 
	that does not converge.
\end{theorem}

\begin{proof}
	We will construct the sequence together with a nested chain of intervals $[B_T, U_T]$ such that $\bigcap_{T=0}^\infty [B_T, U_T] = [B_\infty, U_\infty]$ for some $B_\infty < 1/2 < U_\infty$. Furthermore, there will be infinitely many terms of the sequence $\set{a_n}_{n = 1}^\infty$ lying above $U_\infty$ and below $B_\infty$. 
	
	The sequence and the intervals will be constructed inductively in stages. More precisely, at each stage $T \in \NN_0$, we will maintain an interval $[B_T,U_T]$, an index  $n_T$, and the following invariant 
	\begin{itemize}
	\item if $T$ is even, then 
		\begin{itemize}
		\item at least $\delta_{n_T}^{1/2}$ fraction of terms among $\set{a_n}_{n = 1}^{n_T}$ are above $U_T$; 
		\item at least $\frac{1}{2}$ fraction of terms among $\set{a_n}_{n = 1}^{n_T}$ are below $B_T$. 
		\end{itemize}
	\item if $T$ is odd, then 
		\begin{itemize}
		\item at least $\frac{1}{2}$ fraction of terms among $\set{a_n}_{n = 1}^{n_T}$ are above $U_T$;
		\item at least $\delta_{n_T}^{1/2}$ fraction of terms among $\set{a_n}_{n = 1}^{n_T}$ are below $B_T$. 
		\end{itemize}
	\end{itemize}
	At stage $T = 0$, we initialize with $B_0 = 0$, $U_0=1$, $n_0 = k$, and define the sequence
	\begin{equation*}
		a_1 = \cdots = a_{\floor{k/2}} = 0 \text{ and } a_{\ceil{k/2}} = \cdots = a_k = 1.
	\end{equation*}
	We assume that $k$ is sufficiently large. Suppose now that we are at stage $T$, having defined all terms of the sequence up to index $n_T$. Without loss of generality, assume $T$ is even. For  brevity, set $\varepsilon_* \triangleq \varepsilon_{n_T}$ and $\delta_* \triangleq \delta_{n_T}$.  Define
	\begin{equation*}
		n_{T+1} \triangleq \,n_T /{2\delta_*^{1/2}} \text{ and } a_{n+1} \triangleq  \varepsilon_{*} \bar{a}_n + (1-\varepsilon_*)\bar{a}_{\delta_{*} n}^{(n)} \text{ for } n_T\leq n < n_{T+1}.
	\end{equation*}
	Since $\varepsilon_n \leq \varepsilon_{*}$ and $\delta_n \leq \delta_*$, this defines a valid extension of the sequence.
	\medskip
	
	To define the new endpoints, observe that the fraction of terms among $\set{a_n}_{n=1}^{n_T}$ that are less than $B_T$ is at least $\frac{n_T/ 2}{n_{T+1}}= \delta_*^{1/2}$. We may therefore set $B_{T+1} \triangleq B_T$. Furthermore, by \cref{prop:technical-divergence}, at least half of the terms among $\set{a_n}_{n=1}^{n_{T+1}}$ exceed $(1-C\varepsilon_{n_T}\delta_{n_T}^{1/2})U_T \triangleq U_{T+1}$ for a sufficiently large constant $C$. This concludes the construction. 
	
	Observe that, by construction, there are  infinitely many terms of the sequence $\set{a_n}_{n = 1}^\infty$ lying above $U_\infty$ and below $B_\infty$. It therefore suffices to show $U_\infty > B_\infty$. This follows directly from \cref{prop:divergence-contraction} since  $\alpha + \beta / 2 > 1$ and	 \begin{equation*}
	 	B_\infty \triangleq \lim_{T\to \infty} B_T  < \frac{1}{2} < \lim_{T \to \infty} U_T  \triangleq U_\infty.
	 \end{equation*}
	 This concludes the proof.
	 \end{proof}

\section{Fixed Shape Analysis}\label{sec:fixed-shape}

The main results of this paper focus on worst-case (possibly adversarial) non-stationary weights that are constrained between two envelopes. In contrast, we may
 encounter weights that are obtained by discretizing a limiting
``shape'' on $(0,1)$.

This setting is largely orthogonal to our envelope framework: discretized shapes may assign polynomially large weight to the very recent past, which is far beyond the polylogarithmic ceilings in \cref{thm:main}. Yet, the additional structure allows a different proof
strategy based on renewal theory.

\subsection{A Renewal Lemma for Multiplicative Recursions}

In this subsection, we develop a general and useful lemma that will be used to obtain the main result of the section.

\begin{definition}
A real-valued random variable $Y$ is \emph{lattice} if there exist constants $a\in\RR$ and
$d>0$ such that $\PP(Y\in a+d\ZZ)=1$. Otherwise $Y$ is \emph{non-lattice}.
\end{definition}

\begin{definition}[DRI]
A real-valued function $\eta : \RR_+ \to \RR$ is said to be \emph{directly Riemann integrable} if
\begin{equation*}
    \lim_{h \downarrow 0} L(h) = \lim_{h \downarrow 0} U(h),
\end{equation*}
where $L$ and $U$ are the lower and upper mesh sums, defined by
\begin{align*}
    L(h) &= h \sum_{k = 0}^\infty \inf\set{\eta(t) : t \in (kh,(k+1)h]} \\
    U(h) &= h \sum_{k = 0}^\infty \sup\set{\eta(t) : t \in (kh,(k+1)h]}.
\end{align*}
\end{definition}

\begin{lemma}[Renewal lemma with integrable error]\label{lem:renewal-with-error}
Let $T$ be a random variable taking values in $(0,1)$, and let
$F:\RR_{+}\to\RR$ be a bounded piecewise continuous function satisfying
\begin{equation}\label{eq:mult-renewal}
\qquad
F(x) = \EE\bigl[F(Tx)\bigr] + \epsilon(x)
\quad \text{for all } x \ge 1.
\end{equation}
Assume that $Y\triangleq \log(1/T)$ is non-lattice with finite mean $\mu\triangleq \EE[Y] < \infty$,
and that $|\epsilon(x)|/x$ is directly Riemann integrable over $(1,\infty)$.
Then, $F(x)$ converges and
\begin{equation}\label{eq:mult-renewal-limit}
\lim_{x\to\infty} F(x)
= \EE\bigl[F(\widetilde T)\bigr]
+ \frac{1}{\mu}\int_1^\infty \frac{\epsilon(x)}{x}\,dx,
\end{equation}
where $\widetilde T$ is a random variable on $(0,1)$ with density $\widetilde p(t)= {\PP(T<t)}/{(\mu t)}$.
\end{lemma}

\begin{proof}
Define $G,\eta:\RR\to\RR$ by the logarithmic change of variables
\begin{equation*}
\qquad
G(s)\triangleq F(e^{s})
\quad\text{and}\quad
\eta(s)\triangleq \epsilon(e^{s})
\quad\text{for all } s\in\RR.
\end{equation*}
Then  (by the substitution $x=e^{s}$)  \eqref{eq:mult-renewal} becomes
\begin{equation*}
\qquad
G(s)=\EE\bigl[G(s-Y)\bigr]+\eta(s)
\quad \text{for all } s\ge 0,
\end{equation*}
where $Y=\log(1/T)$ is non-lattice with mean $\mu$.
Moreover, $G$ is bounded and piecewise continuous, and $\eta$ is DRI. Let $\{Y_j\}_{j\ge 1}$ be i.i.d. copies of $Y$, and set $S_0=0$ and
$S_m\triangleq \sum_{i=1}^m Y_i$.
A straightforward induction on $m$ yields
\begin{equation*}
\qquad
G(s)=\EE\bigl[G(s-S_m)\bigr]
+\EE\Bigl[\sum_{n=0}^{m-1}\eta(s-S_n)\Bigr]
\quad \text{for all } m\ge 1.
\end{equation*}
For each $s\ge 0$, define the first passage time
$\tau_s\triangleq \inf\{m\ge 1: S_m\ge s\}$ and the overshoot
$R_s\triangleq S_{\tau_s}-s\in[0,\infty)$. We will now show that 
\begin{equation}\label{eq:renewal-ost}
\qquad
G(s)=\EE\bigl[G(-R_s)\bigr]
+\EE\Bigl[\sum_{n=0}^{\tau_s-1}\eta(s-S_n)\Bigr].
\end{equation}
To this end, let $(\calF_{m})_{m  = 0}^\infty$ be the natural filtration of the process $(Y)_{m = 0}^\infty$.  Fix $s \geq 0$ and define $M_0 = 0$ and 
\begin{equation*}
    M_m = G(s - S_{m \wedge \tau_s}) + \sum_{n = 0}^{m \wedge \tau_s - 1}\eta(s- S_n).
\end{equation*}
Observe that since $\eta(s) = G(s) - \EE\sqb{G(s-Y)}$ and $G$ is bounded, $\norm{\eta}_\infty \leq 2\norm{G}_\infty$. This immediately implies the integrability of $M_m$. Furthermore, it is easy to check that $\EE[M_{m+1} | \calF_m] = M_m$, so $(M_m)_{m = 0}^\infty$ is a martingale. Since $\tau_s \wedge m$ is a bounded stopping time, the optional stopping theorem applies so 
\begin{equation*}
    G(s) = \EE[G(s - S_{\tau_s \wedge m})] + \EE\Bigl[\sum_{n=0}^{m \wedge \tau_{s} - 1}\eta(s-S_n)\Bigr].
\end{equation*}
Note that since $Y > 0$ we have $\EE[\tau_s] < \infty$. Since $G$ and $\eta$ are bounded, 
\begin{equation*}
    |G(s-S_{\tau_s \wedge m})| \leq \norm{G}_\infty \quad \text{ and } \quad \sum_{n=0}^{m \wedge \tau_{s} - 1}\eta(s-S_n) \leq \norm{\eta}_\infty \tau_s.
\end{equation*}
Hence, by dominated convergence theorem and the fact that $\tau_s \wedge m \to \tau_s$ almost surely, we deduce \eqref{eq:renewal-ost}. We now proceed to analyze equation \eqref{eq:renewal-ost}.

\smallskip
\noindent\emph{Overshoot Term.}
Since $Y$ is non-lattice with finite mean $\mu$, the excess life convergence theorem (section 10.3 of \cite{grimmett2020probability})  implies that $R_s$ converges in distribution to a random variable $R$ with density $r\mapsto \PP(Y>r)/\mu$ on $[0,\infty)$; Hence, 
\begin{equation*}
\lim_{s\to\infty}\EE\bigl[G(-R_s)\bigr]=\EE\bigl[G(-R)\bigr].
\end{equation*}
If we define $\widetilde T\triangleq e^{-R}$, then $G(-R)=F(\widetilde T)$.
A change of variables shows that $\widetilde T$ has density
$\widetilde p(t)=\PP(T<t)/(\mu t)$ on $(0,1)$, so $\EE[G(-R)]=\EE[F(\widetilde T)]$.

\smallskip
\noindent\emph{Error term.}
Let $\sigma$ be the renewal measure
\begin{equation*}
\sigma([a,b])\triangleq \sum_{n\ge 0}\PP(S_n\in [a,b]) \quad \text{ for all } a\leq b.
\end{equation*}
Then, since $\eta$ is integrable,
\begin{align*}
\EE\Bigl[\sum_{n=0}^{\tau_s-1}\eta(s-S_n)\Bigr]
&=\EE\Bigl[\sum_{n\ge 0}\eta(s-S_n)\,\II\{S_n<s\}\Bigr]
=\int_{[0,s)} \eta(s-u)\,\sigma(du).
\end{align*}
Applying the key renewal theorem (Theorem 35 in \cite{serfozo2009basics} or section 10.2 of \cite{grimmett2020probability}) yields 
\begin{equation*}
\lim_{s\to\infty}\int_{[0,s)} \eta(s-u)\,\sigma(du)
=\frac{1}{\mu}\int_0^\infty \eta(u)\,du
=\frac{1}{\mu}\int_1^\infty \frac{\epsilon(x)}{x}\,dx,
\end{equation*}
where the last equality uses the substitution $x=e^{u}$. Finally, taking the limit of \eqref{eq:renewal-ost} gives \eqref{eq:mult-renewal-limit}.
\end{proof}

\begin{rmk}
The non-lattice condition and the finiteness of $\EE[\log(1/T)]$ are both necessary
for convergence in \cref{lem:renewal-with-error}, even if $\epsilon\equiv 0$.
\end{rmk}

\subsection{Convergence for Fixed Shapes}

In this subsection we formalize what it means for the sequence $\{p_n\}_{n\ge 1}$ to have a "fixed shape", and we prove a sufficient condition for convergence in this regime. 

\begin{definition}\label{def:discretize}
Let $p:(0,1)\to[0,\infty)$ be a continuous probability density and $\set{p_n}_{n  = 1}^\infty$ be a sequence of probability mass functions with $p_n$ supported on $[n]$. We say that $\{p_n\}_{n=1}^\infty$ \emph{strongly discretizes} $p$ if there exist constants $\kappa > 0$ and $C > 0$ such that
\begin{equation*}
    \norm{p_x - p}_1 \leq C x^{-\kappa} \text{ for all } x \geq C,
\end{equation*}
where $ p_x$ is a probability density on $(0,1)$ defined by 
\begin{equation*}
    p_x(t) \triangleq x p_{\floor x}( \ceil {tx}) \text{ for all } t \in (0,1). 
\end{equation*}
\end{definition}

The next theorem shows that the fixed shape structure, together with the quantitative discretization assumption above, is sufficient to guarantee convergence of the lookback averages.

\begin{theorem}[Convergence for Fixed Shapes]\label{thm:discrete-fixed-shape}
Let $p:(0,1)\to[0,\infty)$ be a continuous probability density with finite log-moment $\int_0^1 p(t) \log (1/t)  \, dt$. Suppose that $\{p_n\}_{n=1}^\infty$ strongly discretizes $p$, and that   $\{a_n\}_{n=1}^\infty\subseteq \RR^d$ satisfies
\begin{equation*}
\qquad
 a_{n+1}=\EE_{j\sim p_n}\!\bigl[a_j\bigr] \text{ for all } n \geq k .
\end{equation*}
Then the sequence $\{a_n\}_{n=1}^\infty$ converges.
\end{theorem}

\begin{proof}
It suffices to prove convergence in the case $d=1$.
The idea is to embed the discrete recursion into a continuous equation and apply
\cref{lem:renewal-with-error}.  Define a function $F:\RR_+\to\RR$ by $F(x)\triangleq a_{\lceil x\rceil}$. It suffices to show that $F(x)$ converges as $x \to \infty$.  Define $\set{p_x}_{x > k}$ as in \cref{def:discretize}, and note that for all $n \geq k$ and $x \in (n, n+1]$, 
\begin{align*}
\int_0^1 p_x(t)F(tx)\,dt &= \int_0^1 x \sum_{j = 1}^n \II\set{tx \in (j-1,j]} p_n(j) F(tx) \,dt \\
&= \int_0^x \sum_{j=1}^n \II\{u\in(j-1,j]\}\,p_n(j)\,F(u)\,du \\
&=\sum_{j=1}^n p_n(j)a_j= a_{n+1}=F(x).
\end{align*}
Consequently, 
\begin{equation*}
F(x)
=\int_0^1 p(t)F(tx)\,dt
+\int_0^1\bigl[p_x(t)-p(t)\bigr]F(tx)\,dt
= \EE\bigl[F(Tx)\bigr]+\epsilon(x),
\end{equation*}
where $T\sim p$ and $\epsilon(x) \triangleq \int_0^1[p_x(t) - p(t)]F(tx) \,dt$.

To show that $F$ converges, our goal becomes to verify the conditions of \cref{lem:renewal-with-error}. Note that since $\set{a_n}_{n=0}^\infty$ is bounded, $F$ is a bounded piecewise continuous function. Moreover, since $p$ has a continuous density and finite log-moment, $\log(1/T)$ is non-lattice and has finite mean. Hence, it suffices to prove that  $\epsilon(x)/x$ is directly Riemann integrable over $(1,\infty)$.

By a standard criterion for direct Riemann integrability (Remark 34 in \cite{serfozo2009basics}), to show that $\epsilon(x)/x$ is directly Riemann integrable over $(1,\infty)$, it suffices to exhibit an eventually decreasing integrable function $\Delta(x)$, such that $|\epsilon(x)|/x \leq \Delta(x)$ for all sufficiently large $x$. This follows from the following inequalities
\begin{align*}
    \frac{\abs{\epsilon(x)}}{x} \leq  \frac{1}{x}\int_0^1 \abs{p_x(t)-p(t)}|F(tx)|\,dt \leq \frac{\norm{F}_\infty \norm{p_x-p}_1}{x} \leq C \frac{\norm{F}_\infty}{x^{1+\kappa}}.
\end{align*}
In the last inequality, we used the fact that $\set{p_n}_{n = 1}^\infty$ strongly discretizes $p$. This concludes the proof.

\end{proof}

\begin{rmk}[Comparison to Envelope Bounded Weights]
If $p(t)\asymp t^{-\gamma}$ as $t\downarrow 0$, then the discretization assigns
$p_n(1)\asymp n^{\gamma-1}$, i.e., a polynomial ceiling $c(n)\asymp n^{\gamma}$.
This illustrates that \cref{thm:discrete-fixed-shape} is genuinely outside the polylogarithmic
envelope regime of \cref{thm:main}, and relies crucially on the fixed shape structure.
\end{rmk}

\appendix
\section{Appendix}

\begin{prop}\label{prop:convergence-contraction}
	Let $A,B > 0$ and $\alpha, \beta \geq 0$ be constants satisfying $\alpha + \beta / 2 \leq 1$. 
	Suppose $\varepsilon_n = A (\log n)^{-\alpha}$ and $\delta_n = B(\log n)^{-\beta}$.  Let $\set{(\Delta_T, n_T)}_{T=0}^\infty$  be a sequence with initialization $\Delta_0=1$ and $n_0 = k$, defined by 
		\begin{gather*}
			\Delta_{T+1} \triangleq  \Delta_T \times (1- c\varepsilon_{n_{T+1}}\delta_{n_{T+1}}^{1/2}) \text{ where } 
			n_{T+1} \triangleq  n_T \times \frac{C}{\varepsilon_{n_{T+1}}^2\delta_{n_{T+1}}^3 \Delta_T^2}.
		\end{gather*}
		 Then, $\Delta_T \to 0$ as $T \to \infty$. 
\end{prop}	
\begin{proof} 
	To prove the convergence, it suffices to show that 
	\begin{equation*}
		\sum_{T= 1}^\infty \varepsilon_{n_{T}}\delta_{n_{T}}^{1/2} = \infty.
	\end{equation*}
	Assume contrary that $\Delta_T \to \Delta >0$.  By definition of $n_{T+1}$, we get
	\begin{align*}
		\frac{n_{T+1}}{\log (n_{T+1})^{2\alpha + 3\beta}} \leq  \frac{C}{A^2B^3\Delta^2} \times n_T &\implies n_{T+1} \leq n_T \times \log (n_T)^{O(1)} \\
		&\implies \log n_T \leq O(T \log T).
	\end{align*}
	Consequently,  $\varepsilon_{n_T}\delta_{n_T}^{1/2} = \Omega\paren{[T \log T]^{-\set{\alpha + \frac{\beta}{2}}}}$, which forms a divergent series in $T$ when $\alpha + \beta / 2 \leq 1$. 
	\end{proof}

\begin{prop}\label{prop:divergence-contraction}
	Let $A,B > 0$ and $\alpha, \beta > 0$ be constants satisfying $\alpha + \beta / 2 > 1$. 
	Suppose $\varepsilon_n = A (\log n)^{-\alpha}$ and $\delta_n = B(\log n)^{-\beta}$. Let $\set{(U_T, n_T)}_{T = 0}^\infty$ be a sequence with initialization   $U_0 > 0$   and $n_0 \geq k$, defined by  
	\begin{equation*}
		n_{T+1} = n_T / 2 \delta_{n_T}^{1/2} \text{ and }U_{T+1} = U_T(1-C\varepsilon_{n_T}\delta_{n_T}^{1/2}) \text{ for all } T\geq 0.
	\end{equation*} 
	If $k$ is sufficiently large, then 
	\begin{equation*}
		U_\infty \triangleq  \lim_{T \to \infty}U_T  > \frac{U_0}{2}
	\end{equation*}
\end{prop}
\begin{proof}
	Observe that as long as $U_\infty >0$, we may choose $k$ sufficiently large to  guarantee $U_\infty > U_0 / 2$. It therefore suffices to prove $U_\infty > 0$, which is equivalent to

	\begin{equation*}
		\sum_{T=0}^\infty \varepsilon_{n_T}\delta_{n_T}^{1/2} < \infty.
	\end{equation*}
	It follows from the definition that $n_{T+1} \geq 2n^T$, as long as $k$ is sufficiently large. This implies $\log n_T = \Omega(T)$. In particular, $\varepsilon_{n_T}\delta_{n_T}^{1/2} = O\paren{T^{-\set{\alpha + \frac{\beta}{2}}}}$ forms a convergent series when $\alpha + \beta / 2 > 1$.
\end{proof}

\bibliographystyle{plain}
\bibliography{references}

\end{document}